\documentclass[reqno,12pt]{amsart}
\theoremstyle{plain}
\newtheorem{thm}{Theorem}[section]
\newtheorem{lemma}[thm]{Lemma} 
\newtheorem{prop}[thm]{Proposition}

\theoremstyle{remark}

\newtheorem{remark}[thm]{Remark}

\theoremstyle{definition}
\newtheorem{defi}[thm]{Definition}
\newtheorem{example}[thm]{Example}

\usepackage{amscd,amssymb,comment,epic,eepic,euscript,graphics}
\usepackage{enumerate}
\usepackage[initials]{amsrefs}
\usepackage[all]{xy}

\newcount\theTime
\newcount\theHour
\newcount\theMinute
\newcount\theMinuteTens
\newcount\theScratch
\theTime=\number\time
\theHour=\theTime
\divide\theHour by 60
\theScratch=\theHour
\multiply\theScratch by 60
\theMinute=\theTime
\advance\theMinute by -\theScratch
\theMinuteTens=\theMinute
\divide\theMinuteTens by 10
\theScratch=\theMinuteTens
\multiply\theScratch by 10
\advance\theMinute by -\theScratch

\def\today{{\number\day\space
 \ifcase\month\or
  January\or February\or March\or April\or May\or June\or
  July\or August\or September\or October\or November\or December\fi
 \space\number\year}}

\newcommand\Ac{{\mathcal{A}}}
\newcommand\Afr{{\mathfrak A}}
\newcommand\ah{{\hat a}}

\newcommand\Bc{{\mathcal{B}}}
\newcommand\betat{{\tilde\beta}}
\newcommand\bh{{\hat b}}
\newcommand\Bt{{\widetilde B}}

\newcommand\Cpx{{\mathbf C}}
\newcommand\Dc{{\mathcal{D}}}
\newcommand\Dt{{\widetilde D}}
\newcommand\eps{\epsilon}
\newcommand\Et{{\widetilde E}}
\newcommand\ev{{\operatorname{ev}}}
\newcommand\HEu{{\EuScript H}}                   
\newcommand\Mcal{{\mathcal{M}}}
\newcommand\Nats{{\mathbf N}}
\newcommand\Nc{{\mathcal{N}}}
\newcommand\phit{{\tilde\phi}}

\newcommand\Qc{{\mathcal{Q}}}
\newcommand\Reals{{\mathbf R}}
\newcommand\restrict{{\upharpoonright}}
\newcommand\rhoh{{\hat\rho}}

\newcommand\sigmat{{\tilde\sigma}}
\newcommand\xh{{\hat x}}
\newcommand\yh{{\hat y}}

\begin{document}

\title{Tail algebras of quantum exchangeable random variables}

\author[Dykema]{Kenneth J.\ Dykema$^{*}$}
\address{K.\ Dykema, Department of Mathematics, Texas A\&M University,
College Station, TX 77843-3368, USA}
\email{kdykema@math.tamu.edu}
\thanks{\footnotesize $^{*}$Research supported in part by NSF grant DMS-0901220}

\author[K\"ostler]{Claus K\"ostler}
\address{C.\ K\"ostler,
Institute of Mathematical and Physical Sciences,
Aberystwyth University,
Aberystwyth SY23 3BZ,
Wales, U.K.}
\email{cck@aber.ac.uk}


\date{November 5, 2012}

\begin{abstract}
We show that any countably generated von Neumann algebra with specified normal faithful
state can arise as the tail algebra of a quantum exchangeable sequence of noncommutative random variables.
We also characterize the cases when the state corresponds to a limit of convex combinations of free products states.
\end{abstract}

\maketitle

\section{Introduction}
Exchangeability is a basic distributional symmetry in probability.
It means that the distribution of a sequence of random variables is invariant under finite permutations of these random variables. The de Finetti theorem characterizes an exchangeable infinite sequence of random variables to be identically distributed and conditionally independent over its tail $\sigma$-algebra. An alternative formulation of this famous theorem is that the law of such a sequence is a mixture of infinite product measures,  where the mixture is specified by a certain random probability measure (see e.g.~\cite{Ka05}). Hewitt and Savage~\cite{HS55} cast this theorem in terms of symmetric measures on infinite Cartesian products of a compact Hausdorff space $S$, and inspired St{\o}rmer's transfer~\cite{St69} of their approach to a C*-algebraic setting: symmetric states on an infinite tensor product of a C*-algebra with itself are identified as a mixture of infinite product states. In analogy with the classical situation, here the mixture is specified by a probability measure on the state space of the C*-algebra.

Recently a noncommutative de Finetti theorem was discovered by K\"ostler and Speicher~\cite{KSp09} in the realm of Voiculescu's free probability theory (see~\cite{VDN92}).
Replacing random variables by operators and the role of permutations by the natural coaction of Wang's quantum permutations~\cite{W98}, the notion of a quantum exchangeable sequence was introduced in a framework of noncommutative probability spaces.
In close analogy to the classical case, a quantum exchangeable infinite sequence is characterized as being identically distributed and `conditionally free' over its tail algebra. Here `conditional freeness'  means `freeness with amalgamation'.         

In this note we have a closer look at tail algebras of quantum exchangeable sequences. Our main results are:

\begin{itemize}
\item[$\diamond$] Any countably generated von Neumann algebra may appear as the tail algebra of a quantum exchangeable sequence of selfadjoint operators (see Theorem~\ref{thm:givenN}).
\item[$\diamond$] The tail algebra of a quantum exchangeable sequence lies in the center of the von Neumann algebra generated by the sequence if and only if the corresponding state is a limit of convex combinations of free product states (see Proposition~\ref{prop:centraltail}).   
\item[$\diamond$] There exist quantum exchangeable sequences of projections generating a finite factor
and whose tail algebra is nontrivial and abelian (see Example~\ref{example}).  
Thus, the corresponding state is not a limit of convex combinations of free product states.
\end{itemize}
Altogether these results show that, as to be expected, the structure of tail algebras of quantum exchangeable infinite sequences has a much higher complexity than those of tail $\sigma$-algebras of exchangeable sequences in probability theory. 

\smallskip
\noindent
{\bf Acknowledgment:}
Most of this research was conducted while 
at the Erwin Schr\"od\-inger Institute during the program on Bialgebras in Free Probability;
the authors would like to thank the Institute and the organizers of the program.
They also thank an anonymous referee for helpful comments.

\section{Preliminaries}

A {\em noncommutative probability space} is a pair $(A,\phi)$, where $A$ is a unital algebra
over the complex numbers and $\phi$ is a linear functional on $A$ sending $1$ to $1$,
and the elements of $A$ are called {\em noncommutative random variables}.
A {\em W$^*$--noncommutative probability space} is one in which $A$ is a von Neumann algebra
and $\phi$ is a normal state.
In this paper, we only consider W$^*$--noncommutative probability spaces $(A,\phi)$ where the state $\phi$ is faithful.
The {\em Noncommutative de Finetti Theorem} of K\"ostler and Speicher~\cite{KSp09}, states that
a sequence $(x_i)_{i=1}^\infty$ of noncommutative random variables in a W$^*$--noncommutative probability space is quantum exchangeable if and only if the sequence is free over the tail algebra of the sequence,
with respect to the $\phi$--preserving conditional expectation.
Quantum exchangeability of the sequence is defined in terms of a natural coaction of the quantum permutation group
of Wang~\cite{W98}.
(See~\cite{KSp09} for more on this).
The {\em tail algebra} is the von Neumann algebra $\bigcap_{n=1}^\infty W^*(\{x_i\mid i\ge n\})$
and the $\phi$--preserving conditional expectation $E$ from $W^*(\{x_i\mid i\ge1\})$
onto the tail algebra is guaranteed to exist.
This was proved in~\cite{K10}, see also Proposition 4.2 of~\cite{KSp09}.

The following is the ``hard part'' of the noncommutative de Finetti theorem
(Thm.\ 1.1 of~\cite{KSp09}).
\begin{thm}[\cite{KSp09}]\label{thm:amalgfp}
Let $(x_i)_{i=1}^\infty$
be a quantum exchangeable sequence in the W$^*$--noncommutative probability space $(\Mcal,\phi)$, where $\phi$
is faithful, and suppose $\Mcal$ is generated by $\{x_i\mid i\in I\}$.
Let $\Nc$ be the tail algebra and let $E:\Mcal\to\Nc$ be the $\phi$--preserving
conditional expectation onto $\Nc$.
Let $\Ac_i$ be the von Neumann subalgebra of $\Mcal$ generated by $\Nc\cup\{x_i\}$.
Then the family $(\Ac_i)_{i\in I}$ is free with amalgamation over $\Nc$, with respect to $E$
and, therefore,
\begin{equation}\label{eq:Mfp}
(\Mcal,E)\cong(*_\Nc)_{i=1}^\infty(\Ac_i,E_i)
\end{equation}
is isomorphic to the W$^*$--amalgamated free product of countably infinitely many copies
of $(\Ac_1,E_1)$,
where $E_i:\Ac_i\to\Nc$ is the restriction to $\Ac_i$ of $E$.
\end{thm}

The next result is the ``easy part'' of the noncommutative de Finetti theorem
(Prop.\ 3.1 of~\cite{KSp09}):
\begin{prop}[\cite{KSp09}]\label{prop:free}
Let 
\[
(\Mcal,E)\cong(*_\Bc)_{i=1}^\infty(\Ac_i,E_i)
\]
be an amalgamated free product of von Neumann algebras, where every $E_i$ is faithful and normal.
Let $\phi$ be any normal, faithful state on $\Bc$ and denote also by $\phi$ the state $\phi\circ E$
on $\Mcal$.
If $x_i\in\Ac_i$ is such that the moments $E_i(x_ib_1x_ib_2\cdots b_{n-1}x_i)$
for all $n\in\Nats$ and $b_1,\ldots b_{n-1}\in\Bc$ are independent of $i$,
then the sequence $(x_i)_{i=1}^\infty$ in the noncommutative probability space $(\Mcal,\phi)$
is quantum exchangeable.
\end{prop}

In this note, we will prove (Theorem~\ref{thm:givenN})
that every countably generated von Neumann algebra $\Nc$
with specified normal faithful state $\phi$ can arise as the tail algebra of a quantum exchangeable
sequence in a W$^*$--noncommutative probability space $(\Mcal,\phi_\Mcal)$, in such a way that
the restriction of $\phi_\Mcal$ to $\Nc$ is $\phi$.

We also investigate the question, given a quantum exchangeable sequence $(x_i)_{i=1}^\infty$ in
a W$^*$--noncommutative probability space $(\Mcal,\phi)$,
of when the distribution of $(x_i)_{i=1}^\infty$ under $\phi$ is a limit of convex combinations of equidistributed free
product states.
Naturally enough, this occurs if and only if the tail algebra of the sequence commutes with all $x_i$.
See Section~\ref{sec:fps} for more details.
This may be compared to St\o{}rmer's result~\cite{St69}, that in the commutative context, all symmetric states
are limits of convex combinations of tensor powers.

\section{Examples of tail algebras}

Now we investigate the content of the tail algebra in the construction of Proposition~\ref{prop:free}.

\begin{prop}\label{prop:tail}
In the setting of Proposition~\ref{prop:free},
the tail algebra $\Nc$ of the sequence $(x_i)_{i=1}^\infty$
is a von Neumann subalgebra of $\Bc$
and is, in fact,
the smallest unital von Neumann subalgebra $\Bc_0$ of $\Bc$
satisfying
\begin{equation}\label{eq:Bc0}
E(b_0x^{k_1}b_1x^{k_2}\cdots b_{n-1}x^{k_n}b_n)\in\Bc_0.
\end{equation}
whenever 
$n\in\Nats$, $k_1,\ldots,k_n\in\Nats$, $b_0,b_1,\ldots,b_n\in\Bc_0$ and $x=x_i$.
(Note that the above expression is independent of the choice of $i\in\Nats$.)
\end{prop}
\begin{proof}
If we consider the structure of $L^2(\Mcal,\phi)$, we easily see that the tail algebra $\Nc$
is a von Neumann
subalgebra of $\Bc$.
Indeed, from Voiculescu's free product construction (see~\cite{BD01} for more detailed discussion in the
case of von Neumann algebras),
\begin{equation}\label{eq:fpHil}
L^2(\Mcal,\phi)=L^2(\Bc,\phi)\oplus
\bigoplus_{\substack{n\ge1 \\ i_1,\ldots,i_n\ge1 \\ i_j\ne i_{j+1}}}
F_{i_1}\otimes_\Bc F_{i_2}\otimes_\Bc\cdots\otimes_\Bc F_{i_n}\otimes_{\pi_\phi}L^2(\Bc,\phi),
\end{equation}
where $F_i$ is the Hilbert $\Bc$-module $L^2(\Ac_i,E_i)\ominus\Bc$
and $\pi_\phi$ is the Gelfand--Naimark--Segal representation of $\Bc$ on $L^2(\Bc,\phi)$.
In particular, the image in $L^2(\Mcal,\phi)$ of the von Neumann algebra generated by $\{x_i\mid i\ge N\}$
is contained in
\[
L^2(\Bc,\phi)\oplus
\bigoplus_{\substack{n\ge1 \\ i_1,\ldots,i_n\ge N \\ i_j\ne i_{j+1}}}
F_{i_1}\otimes_\Bc F_{i_2}\otimes_\Bc\cdots\otimes_\Bc F_{i_n}\otimes_{\pi_\phi}L^2(\Bc,\phi)
\]
and the intersection of these spaces is $L^2(\Bc,\phi)$.
We must, therefore, have $\Nc\subseteq\Bc$.

Since $\Nc\subseteq\Bc\cap W^*(\{x_i\mid i\ge1\})$, we have
\begin{equation*}
\Nc\subseteq E\big(W^*(\{x_i\mid i\ge1\})\big).
\end{equation*}
Thus, we have
\begin{equation}\label{eq:Exs}
\Nc\subseteq W^*(\{E(x_{i_1}x_{i_2}\cdots x_{i_n})\mid n\ge1,\,i_1,\ldots,i_n\ge1\}).
\end{equation}
To see that we have equality in~\eqref{eq:Exs},
take $j_1,\ldots,j_N\ge1$ and let $M_1=\min(j_1,\ldots,j_N)$ and $M_2=\max(j_1,\ldots,j_N)$.
Then 
$x_{j_1}x_{j_2}\cdots x_{j_N}=b+y$,
where $b=E(x_{j_1}x_{j_2}\cdots x_{j_N})$ and
where the element $\yh$ of $L^2(\Mcal,\phi)$ corresponding to $y$ belongs to 
\[
\bigoplus_{\substack{1\le n\le N \\ M_1\le i_1,\ldots,i_n\le M_2 \\ i_j\ne i_{j+1}}}
F_{i_1}\otimes_\Bc F_{i_2}\otimes_\Bc\cdots\otimes_\Bc F_{i_n}\otimes_{\pi_\phi}L^2(\Bc,\phi).
\]
Since the sequence $(x_i)_{i\ge1}$ is quantum exchangeable, it is also exchangeable, i.e., the joint moments of this
sequence are invariant under arbitrary permutations of $\Nats$.
Thus, for each $p\ge0$ we have
$E(x_{j_1+p}x_{j_2+p}\cdots x_{j_N+p})=E(x_{j_1}x_{j_2}\cdots x_{j_N})$
and, therefore, 
$x_{j_1+p}x_{j_2+p}\cdots x_{j_N+p}=b+y_p$ where
\begin{equation}\label{eq:+p}
\yh_p\in\bigoplus_{\substack{1\le n\le N \\ M_1+p\le i_1,\ldots,i_n\le M_2+p \\ i_j\ne i_{j+1}}}
F_{i_1}\otimes_\Bc F_{i_2}\otimes_\Bc\cdots\otimes_\Bc F_{i_n}\otimes_{\pi_\phi}L^2(\Bc,\phi).
\end{equation}
Now since the subspaces~\eqref{eq:+p} corresponding to values of $p$ that differ by more than $M_2-M_1$
are orthogonal to each other,
we see that for each $q\in\Nats$, the ergodic averages
\[
\frac1K\sum_{p=q}^{q+K-1}x_{i_1+p}x_{i_2+p}\cdots x_{i_N+p}
\]
converge in $L^2(\Mcal,\phi)$ to 
$b$ as $K\to\infty$.
So $b\in W^*(\{x_i\mid i\ge q\})$ for all $q\ge1$.
Letting $q\to\infty$, we get $b=E(x_{j_1}x_{j_2}\cdots x_{j_N})\in\Nc$.
This proves
\begin{equation}\label{eq:Exs=}
\Nc=W^*(\{E(x_{i_1}x_{i_2}\cdots x_{i_n})\mid n\ge1,\,i_1,\ldots,i_n\ge1\}).
\end{equation}

It remains to show $\Nc=\Bc_0$.
Using~\eqref{eq:Exs=}, we have $E(b_0x^{k_1}b_1x^{k_2}\cdots b_{n-1}x^{k_n}b_n)\in\Nc$
whenever 
$n\in\Nats$, $k_1,\ldots,k_n\in\Nats$, $b_0,b_1,\ldots,b_n\in\Nc$ and $x=x_i$.
By the characterization of $\Bc_0$, this implies $\Bc_0\subseteq\Nc$.
So it suffices to see that we have $E(x_{i_1}x_{i_2}\cdots x_{i_n})\in\Bc_0$ for all $n\ge1$
and $i_1,\ldots,i_n\ge1$.
However, this follows by considering how mixed moments of the free variables $x_1,x_2,\ldots$
can be evaluated in terms the moments of the individual $x_i$.
Speicher's operator--valued cumulants~\cite{Sp98} can be used to give a careful proof of this fact.
Indeed, using the expression of the cumulants in terms of moments and the M\"obius function and using the
defining property~\eqref{eq:Bc0} of $\Bc_0$,
we see that for each individual $x=x_i$, the operator valued cumulant
$\kappa[xb_1,xb_2,\ldots,xb_{n-1},x]$ belongs to $\Bc_0$ for every $b_1,\ldots,b_{n-1}\in\Bc_0$.
Now using the moment cumulant formula for $E(x_{i_1}x_{i_2}\cdots x_{i_n})$ and the vanishing of mixed cumulants,
we obtain $E(x_{i_1}x_{i_2}\cdots x_{i_n})\in\Bc_0$.
\end{proof}

\begin{lemma}\label{lem:aA}
Let $\Nc$ be a countably generated von Neumann algebra
equipped with a normal faithful state $\phi$.
Then there is a von Neumann algebra $\Ac$ containing $\Nc$
as a unital von Neumann subalgebra
and possessing a $\phi$--preserving conditional expectation $E:\Ac\to\Nc$
onto $\Nc$,
and there is a self--adjoint element $a\in\Ac$ with the property
that $\{E(a^k)\mid k\in\Nats\}$ generates $\Nc$ as a von Neumann algebra.
\end{lemma}
\begin{proof}
One easily sees (using the spectral theorem)
that the von Neumann algebra $\Nc$ is also generated by a finite or countable collection $(p_i)_{i\in I}$
of projections.
Let $\Ac=\bigoplus_{i\in I}\Nc$, with $\Nc\subseteq\Ac$
identified with the constant sequences and with the conditional expectation $E$
given by some strictly positive weights $(\alpha_i)_{i\in I}$
that sum to $1$:
\[
E((x_i)_{i\in I})=\sum_i\alpha_ix_i.
\]
We let $a=(\beta_ip_i)_{i\in I}\in\Ac$ for some bounded family $\beta_i\in\Reals$; thus, we have
\[
E(a^k)=\sum_i\alpha_i\beta_i^kp_i.
\]

If $I$ is finite, then by choosing all the $\beta_i$ to be distinct, 
the determinant of the matrix $(\beta_i^j)_{i\in I,\,0\le j<|I|}$, being a Vandermonde
determinant,
is seen to be nonzero.
Thus, we recover $\{p_i\mid i\in I\}$ by taking linear combinations of 
$(E(a^k))_{0\le k<|I|}$.

Suppose $I$ is infinite, identify it with $\Nats_0$ and let $\beta_i=2^{-i}$.
Let $\Dc=W^*(\{E(a^k)\mid k\in\Nats\})$.
Then
\[
\lim_{k\to\infty}E(a^k)=\lim_{k\to\infty}\left(\alpha_0p_0+\sum_{i=1}^\infty\alpha_i2^{-ki}p_i\right)=\alpha_0p_0,
\]
where the convergence is in norm topology, and $p_0\in\Dc$.
Similarly,
\begin{gather*}
\lim_{k\to\infty}2^k(E(a^k)-\alpha_0p_0)=\alpha_1p_1, \\
\lim_{k\to\infty}2^{2k}(E(a^k)-\alpha_0p_0-2^{-1}\alpha_1p_1)=\alpha_2p_2
\end{gather*}
and so on.
Thus, we get $p_0,p_1,p_2,\ldots\in\Dc$ and $\Dc$ is all of $\Nc$.
\end{proof}

\begin{thm}\label{thm:givenN}
Let $\Nc$ be a countably generated von Neumann algebra with a
normal, faithful state $\phi$.
Then there is a quantum exchangeable sequence $x_1,x_2,\ldots$ in some
W$^*$--noncommutative probability space $(\Mcal,\phi_\Mcal)$
whose tail algebra is a copy of $\Nc$ in $\Mcal$, with the restriction of $\phi_\Mcal$
to $\Nc$ being $\phi$.
\end{thm}
\begin{proof}
Let $\Ac\supseteq\Nc$ and $a\in\Ac$ be as from Lemma~\ref{lem:aA}.
For each $j\in\Nats$, let $\Ac_j$ be a copy of $\Ac$ with the $\phi$--preserving
conditional expectation onto
$\Nc$ denoted by $E_j$.
Let
\[
(\Mcal,E)=(*_\Nc)_{j=1}^\infty(\Ac_j,E_j)
\]
be their amalgamated free product of von Neumann algebra and, of course, let $\phi_\Mcal=\phi\circ E$.
Let $x_j$ be the copy of $a$ in the $j$th copy $\Ac_j$ of $\Ac$.
By Proposition~\ref{prop:free}, $(x_j)_{j=1}^\infty$ is a quantum exchangeable sequence.
By Lemma~\ref{lem:aA}, the set
\begin{equation}\label{eq:amomset}
\{E(x_j^k)\mid k\in\Nats\}      
\end{equation}
generates all of $\Nc$.
By Proposition~\ref{prop:tail}, the tail algebra of the sequence $(x_i)_{i\in I}$ is equal to $\Nc$.
\end{proof}

\section{Free product states}
\label{sec:fps}

Given a C$^*$--algebra $A$ with state $\phi$ and with self--adjoint elements $x_i\in A$, ($i\in I$),
we wish to regard the variables $x_i$ abstractly, independently of their realization in $A$.
Consider the $*$--algebra $\Cpx\langle X\rangle=\Cpx\langle X_i\mid i\in I\rangle$
of polynomials in
noncommuting variables $X_i$, with the involution defined so that $X_i^*=X_i$,
and consider the
unital $*$--algebra representation $\alpha:\Cpx\langle X\rangle\to A$ that sends $X_i$ to $x_i$.

\begin{defi}\label{def:fps}
By a {\em free product state} on $\Cpx\langle X\rangle$, we will mean a functional $\psi$ of
$\Cpx\langle X\rangle$
such that $\psi(1)=1$ and the variables $(X_i)_{i\in I}$ are free with respect to $\psi$.
If, in addition, for all $k$ the moments $\psi(X_i^k)$ are independent of $i$, then we say $\psi$
is an {\em equidistributed free product state} on $\Cpx\langle X\rangle$.
We will say that a linear functional $\phi$ of $\Cpx\langle X\rangle$ is 
{\em a limit of convex combinations of uniformly bounded free product states}
(respectively, of uniformly bounded equidistributed
free product states)
if for some choice of constants $C_i>0$,
the functional $\phi$ is the limit in the topology of pointwise convergence on $\Cpx\langle X\rangle$
of convex combinations of
free product states (respectively, of equidistributed free product states) $\psi$ on $\Cpx\langle X\rangle$
that satisfy $|\psi(X_i^k)|\le C_i^k$ for all $i\in I$ and $k\in\Nats$.
\end{defi}

We will need a few lemmas about freeness.
Here is an easy result about freeness over a subalgebra of the center.
\begin{lemma}\label{lem:Dcenter}
Let $D=C(\Omega)$ be a commutative, unital C$^*$--algebra and
let $A$ be a unital C$^*$--algebra with $D$ embedded as a unital C$^*$--subalgebra of the center of $A$,
and suppose $E:A\to D$ is a conditional expectation.
Let $I$ be a set and suppose for every $i\in I$  there is a C$^*$--algebra $A_i\in I$ with $D\subseteq A_i$.
For each $\omega\in \Omega$, consider the character $\ev_\omega:f\mapsto f(\omega)$ of $D$
and let $\rho_\omega=\ev_\omega\circ E$.
Then the family $(A_i)_{i\in I}$ is free (over $D$) with respect to $E$ if and only if for every $\omega\in\Omega$,
the family $(A_i)_{i\in I}$ is free (over $\Cpx$) with respect to $\rho_\omega$.
\end{lemma}
\begin{proof}
First suppose the family is free with respect to $E$, fix $\omega\in\Omega$ and let us show freeness with respect to $\rho_\omega$.
Suppose $a_j\in A_{i_j}\cap\ker\rho_\omega$ for $j=1,\ldots,n$ and $i_j\ne i_{j+1}$ for all $1\le j<n$.
We must show $\rho_\omega(a_1a_2\cdots a_n)=0$.
Let $\eps>0$.
Since $E(a_j)\in C(\Omega)$ vanishes at $\omega$, by Tietze's extension theorem there is $f\in C(\Omega)$
such that $f(\omega)=1$, $\|f\|_\infty=1$ and $\|E(a_j)f\|<\eps$ for all $j$.
Let $f_j=E(a_j)f$.
Then freeness with respect to $E$ implies $E((fa_1-f_1)(fa_2-f_2)\cdots (fa_n-f_n))=0$, so we have
\[
\rho_\omega((fa_1-f_1)(fa_2-f_2)\cdots (fa_n-f_n))=0.
\]
But if $M=1+\eps+\max_j\|a_j\|$, then
\begin{multline*}
|\rho_\omega((fa_1)(fa_2)\cdots(fa_n))|= \\
=|\rho_\omega((fa_1)(fa_2)\cdots(fa_n))-\rho_\omega((fa_1-f_1)\cdots(fa_n-f_n))|\le nM^{n-1}\eps.
\end{multline*}
However, we have
$\rho_\omega(a_1a_2\cdots a_n)=\rho_\omega((f^n)a_1a_2\cdots a_n)=\rho_\omega((fa_1)(fa_2)\cdots(fa_n))$,
so letting $\eps\to0$ finishes the proof of freeness with respect to $\rho_\omega$.

Now suppose the family is free with respect to $\rho_\omega$ for all $\omega$ and let us show freeness with respect to $E$.
Suppose $a_i\in A_{i_j}\cap\ker E$ for $i_j$ as above
and let $d=E(a_1a_2\cdots a_n)$.
For every $\omega$ and every $j$, we have $\rho_\omega(a_j)=0$, so by hypothesis, $d(\omega)=\rho_\omega(a_1a_2\cdots a_n)=0$.
So $d=0$.
\end{proof}

Here is an easy result about conditional expectations and Gelfand--Naimark--Segal (GNS) constructions,
whose proof we include for convenience.
\begin{lemma}\label{lem:Ephi}
Let $A$ be a unital C$^*$--algebra and $B\subseteq A$ a unital C$^*$--subalgebra.
Suppose $E:A\to B$ is a conditional expectation onto $B$.
Suppose $\phi$ is a state on $A$ such that $\phi\circ E=\phi$.
Let $(\pi_\phi,\HEu_\phi)$ be the GNS representation of $A$ arising from $\phi$
and let $a\mapsto\ah$ denote the linear mapping $A\to\HEu_\phi$ arising in the GNS construction.
\begin{enumerate}[(i)]
\item Then there is a self--adjoint projection $P_\phi$ on $\HEu_\phi$ such that $E(x)\hat{\;}=P_\phi\xh$ and
\begin{equation}\label{eq:EP}
\pi_\phi(E(x))P_\phi=P_\phi\pi_\phi(x)P_\phi
\end{equation}
for all $x\in A$;
\item
If $B$ lies in the center of $A$, then
there is a conditional expectation $E_\phi:\pi_\phi(A)\to\pi_\phi(B)$ satisfying $E_\phi(\pi_\phi(x))=\pi_\phi(E(x))$ for all $x\in A$.
\end{enumerate}
\end{lemma}
\begin{proof}
If $b\in B$, $x\in A$ and $E(x)=0$, then
\[
\langle\xh,\bh\rangle_\phi:=\phi(b^*x)=\phi(E(b^*x))=\phi(b^*E(x))=0.
\]
Thus, given $a\in A$ and writing $a=E(a)+(a-E(a))$ we get for the norm in $\HEu_\phi$
\[
\|\ah\|_\phi^2=\|E(a)\hat{\;}\|_\phi^2+\|(a-E(a))\hat{\;}\|_\phi^2
\]
and the map $\ah\mapsto E(a)\hat{\;}$ is an idempotent, linear map that is contractive with respect to $\|\cdot\|_\phi$
and, hence, extends to a self--adjoint projection $P_\phi$ from $\HEu_\phi$ onto $\overline{\{\bh\mid b\in B\}}$.
For $a,x\in A$ we have
\[
\pi_\phi(E(x))P_\phi\ah=(E(x)E(a))\hat{\;}=E(xE(a))\hat{\;}=P_\phi(xE(a))\hat{\;}=P_\phi\pi_\phi(x)P_\phi\ah,
\]
which proves~\eqref{eq:EP}.

For~(ii), assume $B$ is in the center of $A$.
Note that if $b\in B$, then for every $a\in A$ we have
\begin{align*}
\|\pi_\phi(b)\ah\|_\phi^2&=\phi(a^*b^*ba)=\phi(E(a^*b^*ba))=\phi(E(a^*a)b^*b)= \\
&=\phi(b^*E(a^*a)b)=\|\pi_\phi(b)(E(a^*a)^{1/2})\hat{\;}\|_\phi^2,
\end{align*}
so if $\pi_\phi(b)P_\phi=0$, then $\pi_\phi(b)=0$.
Hence, the $*$--homomorphism $\pi_\phi(B)\ni\pi_\phi(b)\mapsto\pi_\phi(b)P_\phi$ is isometric.
Thus, we have, for $x\in A$,
\[
\|\pi_\phi(E(x))\|=\|\pi_\phi(E(x))P_\phi\|=\|P_\phi\pi_\phi(x)P_\phi\|\le\|\pi_\phi(x)\|
\]
and $\pi_\phi(x)\mapsto\pi_\phi(E(x))$ is a contractive, linear, idempotent map from $\pi_\phi(A)$ onto $\pi_\phi(B)$;
this is the desired conditional expectation $E_\phi$.
\end{proof}

\begin{remark}
In the above lemma, it is not true that faithfulness of $E$ implies faithfulness of $E_\phi$.
For example, let $A=M_2(C([0,1])$, let $B$ be the center of $A$, identified with $C([0,1])$ and
let $E:A\to B$ be the conditional expectation given by,
for $a=(a_{ij})_{1\le i,j\le 2}\in A$
with $a_{ij}\in C([0,1])$,
\[
E(a)(t)=ta_{11}(t)+(1-t)a_{22}(t),\qquad(t\in[0,1]).
\]
Then $E$ is faithful.
Letting $\phi$ be the state on $A$ determined by $\phi=\phi\circ E$ and $\phi(b)=b(0)$ for $b\in B$, we find that
$E_\phi$ is the state $M_2(\Cpx)\to\Cpx$ sending $\left(\begin{smallmatrix}1&0\\0&0\end{smallmatrix}\right)$ to $0$, so is not faithful.
\end{remark}

\begin{lemma}\label{lem:quotientfree}
Let $A$ be a unital C$^*$--algebra and $B$ a unital C$^*$--subalgebra of the center of $A$ with a conditional expectation $E:A\to B$.
Let $I$ be a set and suppose for every $i\in I$, $A_i$ is a C$^*$--subalgebra of $A$ that contains $B$, and the family $(A_i)_{i\in I}$
is free (over $B$) with respect to $E$.
Let $\phi$ be a state on $A$ satisfying $\phi\circ E=\phi$ and let $E_\phi:\pi_\phi(A)\to\pi_\phi(B)$ be the conditional 
expectation from Lemma~\ref{lem:Ephi}.
Then the family $(\pi_\phi(A_i))_{i\in I}$ is free (over $\pi_\phi(B)$) with respect to $E_\phi$.
\end{lemma}
\begin{proof}
We have $B=C(X)$ for some compact Hausdorff space $X$, and $\pi_\phi(B)=C(Y)$ for a closed subspace $Y$ of $X$,
where the $*$--homomorphism $\pi_\phi\restrict_B:C(X)\to C(Y)$ sends a function to its restriction to $Y$.
For $y\in Y$, let $\ev_y:\pi_\phi(B)\to\Cpx$  and $\ev_y^{(X)}:B\to\Cpx$ be the homomorphisms of evaluation at $y$.
By Lemma~\ref{lem:Dcenter}, it will suffice to show that for every $y\in Y$, the family
$(\pi_\phi(A_i))_{i\in I}$ is free (over $\Cpx$) with respect to $\ev_y\circ E_\phi$.
However, we have 
\[
\ev_y\circ E_\phi\circ\pi_\phi=\ev^{(X)}_y\circ E.
\]
From Lemma~\ref{lem:Dcenter}, we have freeness of $(A_i)_{i\in I}$ with respect to $\ev^{(X)}_y\circ E$, and from this
follows the freeness of
$(\pi_\phi(A_i))_{i\in I}$ with respect to $\ev_y\circ E_\phi$.
\end{proof}

The following result shows that a state that is a limit of convex combinations of uniformly bounded free product
states always arises from the situation of a free product with amalgamation over a subalgebra of the center.
\begin{prop}\label{prop:limfreeprodst}
Let $A$ be a C$^*$--algebra with faithful state $\phi$.
Suppose $x_i\in A$ ($i\in I$), for $I$ countable,
are self--adjoint elements that together generate $A$ as a C$^*$--algebra
and so that $\phit:=\phi\circ\alpha$
is a limit of convex combinations of uniformly bounded free product states, 
where $\alpha:\Cpx\langle X\rangle\to A$
is the $*$--homomorphism given by $X_i\mapsto x_i$.
Then there is a unital C$^*$--algebra $B$ with a unital C$^*$--subalgebra $D\subseteq Z(B)$
of the center of $B$ and with a faithful conditional expectation $E:B\to D$, 
and there is a unital $*$--homomorphism $\pi:A\to B$
and a faithful tracial state $\sigma=\sigma\circ E$ on $B$ so that
$\phi=\sigma\circ\pi$ and the family $(\pi(x_i))_{i\in I}$ is free with respect to $E$.

Furthermore, if $\phit$ is a limit of convex combinations
of uniformly bounded equidistributed free product states with respect to $(x_i)_{i\in I}$,
then $E:B\to D$ and $\pi:A\to B$
may be chosen so that for all $n$ and all $d_0,\ldots,d_n\in D$, the corresponding moment
$E(d_0\pi(x_i)d_1\cdots\pi(x_i)d_n)$
of the variable $\pi(x_i)$ is independent of $i$.
\end{prop}
\begin{proof}
We fix constants $C_i$ as in Definition~\ref{def:fps}.
If $\psi$ is a free product state on $\Cpx\langle X\rangle$, then the usual GNS construction
yields a $*$--representation $\pi_\psi:\Cpx\langle X\rangle\to B(\HEu_\psi)$ with $\|\pi_\psi(X_i)\|\le C_i$.
Let $B_\psi\subseteq B(\HEu_\psi)$ be the unital C$^*$--algebra generated by the image of $\pi_\psi$
and denote also by
$\psi$ the state on $B_\psi$ so that $\psi\circ\pi_\psi$ is the original functional $\psi$ on $\Cpx\langle X\rangle$.
Of course, the condition that $\psi$ is a free product state on $\Cpx\langle X\rangle$ implies that the variables
$(\pi_\psi(X_i))_{i\in I}$ are free in $B_\psi$ with respect to $\psi$.
Moreover, $B_\psi$ is isomorphic to the free product over the scalars of the abelian C$^*$--algebras generated
by the $\pi_\psi(X_i)$.
Since the restrictions of $\psi$ to these abelian C$^*$--algebras are faithful, and since the free product of faithful
states is faithful~\cite{D98} and since the free product of traces is a trace (see~\cite{VDN92}),
it follows that $\psi$ is a faithful trace on $B_\psi$.

By hypothesis, 
we may write $\phit$ as the limit of a net of convex combinations of free product states, and
consider the set $F$ of all the free product states appearing with nonzero coefficients in these convex combinations.
Let
\[
\Bt=\prod_{\psi\in F}B_\psi=\{(b_\psi)_{\psi\in F}\mid b_\psi\in B_\psi,\,\sup_{\psi}\|b_\psi\|<\infty\}
\]
be the C$^*$--algebra direct product.
Then $\betat:\Cpx\langle X\rangle\to\Bt$ defined by $\betat(p)=(\pi_\psi(p))$
is a $*$--representation.
Let $\Dt\subseteq\Bt$ be the subalgebra consisting of all sequences of scalars, i.e.\
all $(\lambda_\psi1)_{\psi\in F}$ for $\lambda_\psi\in\Cpx$.
Clearly, $\Dt\cong\ell^\infty(F)$ is a subalgebra of the center of $\Bt$, and the map $\Et:\Bt\to\Dt$ given by
$\Et((b_\psi)_{\psi\in F})=(\psi(b_\psi))_{\psi\in F}$ is a conditional expectation that is faithful
because each tracial state $\psi$ is faithful on $B_\psi$.
The variables $(\betat(X_i))_{i\in I}$ are free with respect to $\Et$,
because if for some $n\in\Nats$ and $i(1),\ldots,i(n)\in\Nats$ with $i(j)\ne i(j+1)$,
$p_j(X_{i(j)})$ is a polynomial in $X_{i(j)}$ with coefficients from $\Bt$ and $E(p_j(X_{i(j)}))=0$,
then $\psi(p_j(X_{i(j)}))=0$ for every $\psi\in F$;
by freeness $\psi(p_1(X_{i(1)})\cdots p_n(X_{i(n)}))=0$ for every $\psi\in F$, so also
$\Et(p_1(X_{i(1)})\cdots p_n(X_{i(n)}))=0$.
Note that $\Et$ has the property $\Et(xy)=\Et(yx)$ for all $x,y\in\Bt$, owing to the fact that each $\psi$ is a trace.

In the above construction, if all free product states $\psi\in F$ are equidistributed,
then the moments
\begin{equation}\label{eq:Etmoms}
\Et(d_0\betat(X_i)d_1\cdots\betat(X_i)d_n)
\end{equation}
of the variables $\betat(X_i)$ are independent of $i$.

To a convex combination $\rho=\sum t_\psi\psi$ (with finite support) of elements of $F$,
consider the state $\rhoh$ of $\Dt$ given by the corresponding weighted average,
$\rhoh:(\lambda_\psi1)_{\psi\in F}\mapsto\sum t_\psi\lambda_\psi$.
Let $(\rho_j)_{j\in J}$ denote a net of convex combinations of free product states on $\Cpx\langle X\rangle$
that converges (in the topology of pointwise convergence on $\Cpx\langle X\rangle$) to $\phit$.
Replacing this net by a subnet, if necessary, we may without loss of generality assume
that the corresponding net $\rhoh_j$ converges in the weak$^*$--topology on $\Dt^*$
to a state  $\sigmat$ on $\Dt$.
Then we have $\phit=\sigmat\circ\Et\circ\betat$.

We are almost done, except that $\sigmat\circ\Et$ need not be faithful.
Let $\pi_{\sigmat\circ\Et}$ denote the GNS representation associated to the tracial state
$\sigmat\circ\Et$ on $\Bt$, and let $B=\pi_{\sigmat\circ\Et}(\Bt)$ and $D=\pi_{\sigmat\circ\Et}(\Dt)$.
Then the corresponding state $\sigma$ on $B$ is a faithful trace.
By Lemma~\ref{lem:Ephi}, there is a conditional expectation $E:B\to D$ so that
\[
E(\pi_{\sigmat\circ\Et}(x))=\pi_{\sigmat\circ\Et}(\Et(x)),\qquad(x\in\Bt)
\]
and by Lemma~\ref{lem:quotientfree} and freeness of $(\betat(X_i))_{i\in I}$ with respect to $\Et$,
the family
\begin{equation}\label{eq:pbtfam}
(\pi_{\sigmat\circ\Et}(\betat(X_i)))_{i\in I}
\end{equation}
is free (over $D$) with respect to $E$.

Let $\beta=\pi_{\sigmat\circ\Et}\circ\betat:\Cpx\langle X\rangle\to B$.
Then $\phit=\phi\circ\alpha=\sigma\circ\beta$.
We will now define a $*$--homomorphism $\pi:A\to B$
so that all triangles in the diagram
\begin{equation*}
\xymatrix{
\Cpx\langle X\rangle \ar[r]^{\alpha} \ar  [d]_{\beta}  &A \ar[dl]_{\pi}
       \ar[d]^{\phi} \\
B\ar[r]_{\sigma} & \Cpx
}
\end{equation*}
commute.
Indeed, since $\phi$ and $\sigma$ are faithful, we have
\begin{align*}
\|a\|&=\limsup_{n\to\infty}|\phi((a^*a)^n)|^{1/2n},\qquad(a\in A) \\
\|b\|&=\limsup_{n\to\infty}|\sigma((b^*b)^n)|^{1/2n},\qquad(b\in B).
\end{align*}
But this implies $\|\alpha(p)\|=\|\beta(p)\|$ for all $p\in\Cpx\langle X\rangle$.
So the $*$--homomorphism defined on the image of $\alpha$ by $\alpha(p)\mapsto\beta(p)$
is isometric and extends to an isometric $*$--homomorphism $\pi:A\to B$, as required.
Since $\pi(x_i)=\beta(X_i)$,
by freeness of the family~\eqref{eq:pbtfam} we have that
$(\pi(x_i))_{i\in I}$ is free with respect to $E$.

In the case that all free product states are equidistributed, the observation above regarding the moments~\eqref{eq:Etmoms}
of $\betat(X_i)$ with respect to $\Et$ implies that the moments of $\pi(x_i)$ with respect to $E$ are
independent of $i$.
\end{proof}

\begin{prop}\label{prop:centraltail}
Let $(x_i)_{i\in I}$ be quantum exchangeable random variables in $(\Mcal,\phi)$,
suppose $\Mcal$ is generated by $\{x_i\mid i\in I\}$ and
let $\alpha:\Cpx\langle X\rangle\to\Mcal$ be the $*$--homomorph\-ism given by $\alpha(X_i)=x_i$.
Then $\phi\circ\alpha$ is a limit of convex combinations of uniformly bounded equidistributed free product states
with respect to $(x_i)_{i\in I}$
if and only if the tail algebra
$\Nc$ lies in the center of $\Mcal$.
\end{prop}
\begin{proof}
Suppose the tail algebra lies in the center of $\Mcal$.
Then by Theorem~\ref{thm:amalgfp}, $\Mcal$ is
the free product with amalgamation over the tail algebra $\Nc$, and $\phi=\phi\restrict_\Nc\circ E$,
where $E:\Mcal\to\Nc$ is the $\phi$--preserving conditional expectation with respect to which the
algebras $\Ac_i=W^*(\Nc\cup\{x_i\})$ are free over $\Nc$.
Regarding $\Nc$ as a commutative C$^*$--algebra, by the Gelfand theorem, we have $\Nc\cong C(\Omega)$.
Then, by a classical result,
every state on $C(\Omega)$ lies in the closed convex hull of the set of
point evaluation maps $\{\ev_\omega\mid\omega\in\Omega\}$.
By Lemma~\ref{lem:Dcenter}, every functional $\ev_\omega\circ E\circ\alpha$ is an equidistributed
free product state
on $\Cpx\langle X\rangle$
and clearly the boundedness criterion is satisfied with constants $C_i=\|x_i\|$.
Taking convex combinations of the $\ev_\omega$ that approximate $\phi\restrict_\Nc$, we easily see
that $\phi\circ\alpha$ is a limit of convex combinations of equidistributed free product states.

Conversely, suppose that $\phi\circ\alpha$ is a limit of convex combinations of uniformly bounded
equidistributed free product states.
Consider the C$^*$--subalgebra $C^*(\{x_i\mid i\in I\})$
of $\Mcal$ generated by the $x_i$.
By Proposition~\ref{prop:limfreeprodst}, there is a unital C$^*$--algebra $B$ and a unital subalgebra $D$ of the
center of $B$ with a faithful conditional expectation $E:B\to D$ and
a unital, injective $*$--homomorphism $\pi:C^*(\{x_i\mid i\in I\})\to B$ so that the elements $(\pi(x_i))_{i\in I}$ are
free with respect to $E$, and so that the moments of $\pi(x_i)$ are independent of $i$;
furthermore, there is a faithful state $\sigma$ on $D$ so that $\sigma\circ E\circ\pi=\phi\restrict_{\Afr}$.
Let $\Afr$ be the C$^*$--subalgebra of $B$ generated by $D\cup\{\pi(x_i)\mid i\in I\}$.
Then $\Afr$ is an amalgamated free product of C$^*$--algebras $C^*(D\cup\{x_i\})$, with amalgamation over $D$.
We take the Hilbert space representation $\rho$ of $\Afr$ that is the GNS
construction for the restriction of the state $\sigma\circ E$ to $\Afr$.
The strong--operator--topology closure of $\rho(\Afr)$
is a von Neumann algebra, $\Qc$, that is isomorphic to a free product
\begin{equation}\label{eq:overDc}
(*_\Dc)_{i\in I}(W^*(\Dc\cup\{\rho(x_i)\},\overline{E})
\end{equation}
with amalgamation over
the strong--operator--topology closure $\Dc$ of $\rho(D)$.
Taking strong--operator--topology limits (using Kaplansky's density theorem), we easily see that $\Dc$
lies in the center of $\Qc$.
Since $\sigma\circ E$ and $\phi$ are faithful, the composition $\rho\circ\pi$, when compressed to a Hilbert
subspace, equals the GNS construction of the restriction of $\phi$ to $\Afr$.
Therefore, this mapping of $\Afr$ into the strong--operator--closure of $\Afr$ is canonically isomorphic
to the inclusion of $\Afr$ in $\Mcal$, and we may regard $\Mcal$ as embedded in the amalgamated
free product~\eqref{eq:overDc}.
By Proposition~\ref{prop:tail}, the tail algebra of $\{x_i\mid i\in I\}$ lies in $\Dc$.
Thus, the tail algebra is in the center of $\Mcal$.
\end{proof}

The following easy example shows that the tail algebra can be commutative without lying in the center of the algebra
generated by the quantum exchangeable sequence.

\begin{example}\label{example}
Let $\Bc=\Cpx\oplus\Cpx$ embed into $M_2(\Cpx)$ as the diagonal matrices.
For each $i\in\Nats$, let $\Ac_i$ be a copy of $M_2(\Cpx)$ and let $E_i:\Ac_i\to\Bc$ be 
the conditional expectation taking a matrix to its diagonal.
Let
\[
(\Mcal,E)=(*_\Bc)_{i=1}^\infty(\Ac_i,E_i)
\]
be the amalgamated free product of von Neumann algebras.
We note that $\Mcal$ is easily seen to be isomorphic to the free group factor $L(\mathbf{F}_\infty)$.
Let $\psi$ be any faithful state  on $\Bc$.
Then $\phi=\psi\circ E$ is a normal faithful state on $\Mcal$.
Fix $0<t<1/2$ and let $x_i$ be the copy of the projection
$\left(\begin{smallmatrix} t & \sqrt{t(1-t)} \\ \sqrt{t(1-t)} & 1-t \end{smallmatrix}\right)$
in $\Ac_i$.
By Proposition~\ref{prop:free}, the sequence $(x_i)_{i=1}^\infty$ is quantum exchangeable.
Applying Proposition~\ref{prop:tail}, we see that the tail algebra $\Nc$ of the sequence is $\Bc\cong\Cpx\oplus\Cpx$
and, incidentally, the von Neumann algebra generated by the sequence $\{x_i\mid i\in\Nats\}$
is all of $\Mcal$.
\end{example}

\end{document}